\documentclass[10pt]{amsart}
\usepackage[leqno]{amsmath}
\usepackage{amsthm}
\usepackage{amsfonts}
\usepackage[T1]{fontenc}

\usepackage{color}

\usepackage[margin={3cm,4cm}]{geometry}

\theoremstyle{plain}
\newtheorem{lem}{Lemma}[section]
\newtheorem{thm}{Theorem}
\newtheorem{prop}[lem]{Proposition}

\theoremstyle{remark}
\newtheorem{rem}[lem]{Remark}

\newcommand{\ud}{\mathrm{d}}

\providecommand{\abs}[1]{\left\lvert#1\right\rvert} 
\providecommand{\norm}[1]{\left\lVert#1\right\rVert}
 
\numberwithin{equation}{section}

\DeclareMathOperator{\Div}{div}

\DeclareMathOperator{\Dt}{\frac{\ud}{\ud t}}

\title{Strong solutions to the Navier-Stokes equations on thin 3D domains}

\begin{document}

\author{Bernard Nowakowski}
\author{Wojciech M. Zaj\k aczkowski}
\thanks{Research of both authors is partially supported by Polish KBN grant NN 201 396937}
\address{Bernard Nowakowski\\ Institute of Mathematics\\ Polish Academy of Sciences\\ \'Snia\-deckich 8\\ 00-956 Warsaw\\ Poland}
\email{bernard@impan.pl}
\address{Wojciech M. Zaj\k aczkowski\\ Institute of Mathematics\\ Polish Academy of Sciences\\ \'Sniadeckich 8\\ 00-956 Warsaw\\ Poland\\ and \\ Institute of Mathematics and Crypto\-logy\\ Military University of Technology\\ Kaliskiego 2\\ 00-908 Warsaw\\ Poland}
\email{wz@impan.pl}

\maketitle

\begin{abstract}
	We prove the existence of strong solutions to Navier-Stokes equations in three dimensional thin domains. Our proof is based on the energy and the Poincar\'e inequalities as well as contraction principle argument and is free of the mean value operator. The price we pay for the simplicity of the proof are stronger assumptions on the initial velocity and the forcing term. We need to assume that their derivatives with respect to time belong to certain Lebesgue space.
\end{abstract}

\section{Introduction}

We consider the initial-boundary value problem for the Navier-Stokes equations
\begin{equation}\label{eq1}
	\begin{aligned}
		&v_{,t} + v\cdot \nabla v - \nu \triangle v + \nabla p = f & &\text{in $\Omega_\epsilon\times(t_0,T) =: \Omega^T_{\epsilon}$},\\
		&\Div v = 0 & &\text{in $\Omega_\epsilon^T$}, \\
		&v\vert_{t = t_0} = v(t_0) & &\text{in $\Omega_\epsilon$}.
	\end{aligned}
\end{equation}
The domain $\Omega_\epsilon \subset \mathbb{R}^3$ is assumed to be bounded or unbounded. The subscript $\epsilon$ indicates that the domain in the introduced Cartesian system of coordinates $(x_1,x_2,x_3)$ is thin along $x_3$ direction. We consider two types of boundary conditions on $\partial \Omega^T_{\epsilon}$: the pure Dirichlet condition $v = 0$ or a mixture of periodic and the Dirichlet conditions. A detailed description is given later in this Section.

The study of the Navier-Stokes equations in three dimensions has a long and rich history. It was proved by J. Leray in 1933 that problem \eqref{eq1} has at least one weak solution $v$ in $\mathbb{R}^3$, i.e.
\begin{equation*}
	v \in L_{\infty}(0,T; L_2(\mathbb{R}^3)) \cap L_2(0,T;H^1(\mathbb{R}^3)).
\end{equation*}
Later, in 1952 his result was generalized by E. Hopf to bounded domains in $\mathbb{R}^3$. Since then the research was focused on the problem of uniqueness or regularity of weak solutions. The first results in this direction are existence of regular solutions in the two dimensional case, which was proved by O.A. Ladyzhenskaya in 1959, and in the axially symmetric case without swirl, which was shown independently by O.A. Ladyzhenskaya and M.R. Uhovskij, V.I. Yudovich in 1968. Although many ideas have been demonstrated and various approaches have been suggested, the problem still remains open.

In this paper we also limit our considerations to a particular case when the magnitude of domain is small in one direction. The aim of this paper is to prove the existence of regular solutions:

\begin{thm}\label{t1}
	Suppose that
	\begin{equation*}
		\begin{aligned}
			v(t_0) &\in L_2(\Omega';L_{p_1}(0,\epsilon)) \cap H^1(\Omega_{\epsilon}) & p_1 &> 2, \\
			v_{,t}(t_0) & \in L_2(\Omega_{\epsilon}), & &\\
			f_{,t} &\in L_2(t_0,T;L_{\frac{6}{5}}(\Omega_{\epsilon})), \\
			f &\in L_2(t_0,T;L_{\frac{6}{5}}(\Omega');L_{p_2}(0,\epsilon)) & p_2 &> \frac{6}{5}.
		\end{aligned}
	\end{equation*}
	Then, for sufficiently small $\epsilon$, i.e.
	\begin{equation*}
		\epsilon \sim \frac{1}{\norm{v(t_0)}_{L_2(\Omega';L_{p_1}(0,\epsilon))} + \norm{f}_{L_2(t_0,T;L_{\frac{6}{5}}(\Omega');L_{p_2}(0,\epsilon))}}
	\end{equation*}
	any solution $v$ to problem \eqref{eq1} satisfies
	\begin{multline*}
		\sup_{t\in (t_0,T)} \norm{v_{,t}(t)}^2_{L_2(\Omega_\epsilon)} + \nu \norm{\nabla v}^2_{L_{\infty}(t_0,T;L_2(\Omega^t))} \\
		 \leq c_{\nu,\Omega'} \norm{f_{,t}}_{L_2(t_0,T;L_{\frac{6}{5}}(\Omega_\epsilon))}^2 + \norm{v_{,t}(t_0)}^2_{L_2(\Omega_\epsilon)} + c_{\nu,\Omega'}\left(\norm{v(t_0)}^2_{H^1(\Omega_\epsilon)} + \norm{f}_{L_2(t_0,T;L_{\frac{6}{5}}(\Omega_\epsilon))}^2\right),
	\end{multline*}
	where the subscripts in the constants indicate what the constants are dependent of.
\end{thm}

\begin{rem}
	Note that the constants which appear on the right-hand side in the estimate in Theorem \ref{t1} do not depend on time. Therefore the solution to problem \eqref{eq1} can be regarded as global in time. 
\end{rem}

The thin domain approach in solving problem \eqref{eq1} has been used by many authors for $21$ years. The inspiration originated from two papers by J. Hale and G. Raugel \cite{hal1, hal2}, who considered the damped wave and the reaction-diffusion equations on thin domains. It was first adopted by G. Raugel and R. Sell who proved in  \cite{rog1} the existence of global and strong solutions to \eqref{eq1} supplemented with mixed boundary conditions: the periodicity was assumed in the thin direction and the zero Dirichlet boundary condition was set on the lateral boundary. They showed (see \cite[Theorem A]{rog1}) that for $v(0)$ and $f$ satisfying
\begin{equation}\label{eq7}
	\begin{aligned}
		v(0) \in \mathcal{R}(\epsilon) \subset \left\{v \in H^1(\Omega_{\epsilon})\colon \Div v = 0, \int_{\Omega_{\epsilon}} v\, \ud x = 0\right\}, \\
		f \in \mathcal{J}(\epsilon) \subset \left\{f \in W^1_{\infty}([0,\infty);L_2(\Omega_{\epsilon}))\colon \int_{\Omega_{\epsilon}} f\, \ud x = 0\right\},
	\end{aligned}
\end{equation}
where $\mathcal{R}(\epsilon)$ and $\mathcal{J}(\epsilon)$ are certain ``large'' sets, there exists a strong solution $v$ to problem \eqref{eq1} such that
\begin{equation}\label{eq6}
	\norm{v(t)}^2_{H^1(\Omega_{\epsilon})} \leq C < \infty,
\end{equation}
for all $t \geq 0$, where $C = C(v(0),f)$. In addition, there exists a constant $K$ such that
\begin{equation*}
	\limsup_{t\to\infty}\norm{v(t)}^2_{H^1(\Omega_{\epsilon})} \leq K < \infty
\end{equation*}
and the constant $K$ does not depend on $v(0)$.

In their next paper (see \cite{rog2}) they considered the case of the purely periodic boundary conditions and proved (see \cite[Theorem A]{rog2}) that for $v(0)$ as above and 
\begin{equation*}
	f \in \mathcal{S}(\epsilon) \subset \left\{f \in L_{\infty}([0,\infty);L_2(\Omega_{\epsilon}))\colon \int_{\Omega_{\epsilon}} f\, \ud x = 0\right\},
\end{equation*}
where $\mathcal{S}(\epsilon)$ is certain ``large'' set, there exists a (unique) strong solution to problem \eqref{eq1}, which satisfies \eqref{eq6} with $C$ not depending on $v(0)$. 

Their technique is based on the vertical mean operator $M_1$, where
\begin{equation*}
	M_{\epsilon} v(x,y,z) = \frac{1}{\epsilon}\int_0^{\epsilon} v(x,y,s)\, \ud s,
\end{equation*} 
which ensures the decomposition of every function $v$ into $M_1v$ and $N_1 := (I - M_1)v$ with the following properties: $M_1v$ does not depend on the variable along the thin direction and $N_1v$ has zero mean in thin direction. This allows to control more precisely the constants in the Sobolev and Poincar\'e inequalities. 

Subsequently, in the middle of $90$' R. Temam and M. Ziane simplified and generalized the results of G. Raugel and R. Sell to various boundary conditions, which involve the periodic, the free boundary and the Dirichlet boundary conditions (see \cite[Theorem A1]{tem}). Is was possible due to the improved Agmon inequality (see \cite[Prop. 2.1, Cor. 2.2 and Cor. 2.3]{tem}) and thorough control of the constants in Sobolev-type inequalities. Also the characterization of the initial and the external data sets become much more clear but the assumptions remained the same: $v_0$ and $f$ belong to $H^1(\Omega)$ and $L_\infty(0,\infty;L_2(\Omega))$, respectively.

A year later together with I. Moise (see \cite{ion}) they proved (see \cite[Theorem 4.1]{ion}) the existence of global and regular solutions in purely periodic case for much larger class of initial data than in \cite{rog2}. In order to clarify what larger class of the initial and the external data mean we recall that the existence of regular solutions in aforementioned articles is guaranteed as long as
\begin{align*}
	\norm{M_{\epsilon}u(0)}_{H^1(\Omega)} &\leq \epsilon^{p_1} \big(\ln\epsilon)\big)^{q_1} \alpha(\epsilon), & & & \sup_t\norm{M_{\epsilon}f(t)}_{L_2(\Omega)} &\leq \epsilon^{p_2} \big(\ln\epsilon)\big)^{q_2} \alpha(\epsilon), \\
		\norm{N_{\epsilon}u(0)}_{H^1(\Omega)} &\leq \epsilon^{p_3} \big(\ln\epsilon)\big)^{q_3} \alpha(\epsilon), & & & \sup_t\norm{N_{\epsilon}f(t)}_{L_2(\Omega)} &\leq \epsilon^{p_4} \big(\ln\epsilon)\big)^{q_4} \alpha(\epsilon),
\end{align*}
where $p_i,q_i$ for $i \in \{1,2,3,4\}$ are certain real constants and $\lim_{\epsilon\to 0}\alpha(\epsilon) \to 0$. For a brief historical overview of the improvements and comparison of these constants we refer the reader e.g. to Introduction in \cite{mon} and \cite{kuk1}. Let us clarify that ``larger class'' means that the powers $p_i$ and $q_i$ are improved in such way that the norms they bound on the left-hand sides can be in fact much larger.

At the same time J. Avrin \cite{avr} considered the case of purely homogeneous Dirichlet boundary conditions and proved (see \cite[Theorems 1.1 and 1.2]{avr}) the existence of global and strong solutions for large data. 
His choice of the boundary conditions allowed him to use different tools than vertical mean operator. His idea was based on the contraction principle argument (same as ours) and detailed analysis of the dependence of solution on the first eigenvalue of the corresponding Laplace operator. 
It is worth noting that the assumptions of Theorem 1.2 are only $v(t_0) \in L_4(\Omega)$ and $f = f_0 + f_1$, $f_0 \in L_{\infty}(0,\infty;H_{\sigma}(\Omega))$ and $f_1 \in L_q(0,\infty;H_{\sigma}(\Omega))$, where $q > \frac{8}{5}$ and
\begin{equation*}
	H_{\sigma} = \overline{\left\{g \in \mathcal{C}^{\infty}_0\colon \Div g = 0\right\}}^{L_2(\Omega)}.
\end{equation*}

The purely periodic case was also deeply examined at the end of 90'. D. Iftimie (\cite{ift0}) proved the existence and uniqueness of solutions for larger class of data by the means of anisotropic Sobolev spaces and the Littlewood-Paley theory. He required that $f \equiv 0$ and 
$M_{\epsilon}u_0 \in L_2(\mathbb{T}^2)$, $N_{\epsilon}u_0' := (N_{\epsilon}u_1(0),N_{\epsilon}u_2(0)) \in H^{\delta}(\mathbb{T}^3)$ and $N_{\epsilon}u_3(0) \in H^{\frac{1}{2} - \delta}$ for $0 < \delta < 1$ (see also Introduction in \cite{mon}). On the other hand S. Montgomery-Smith (\cite[Theorem 1]{mon}) improved the result from \cite{ion}

Through the first decade in the 21st century further improvements of the powers $p_i$ and $q_i$ were made (see e.g. \cite{ift}, \cite{ift2}, \cite{kuk2}, \cite{kuk1}). In \cite{ift} also weaker restriction on the data was imposed, namely $A^{\frac{1}{4}}N_{\epsilon}v_0$ was assumed to belong to certain subspace of $L_2(\Omega)$, where $A$ is the Stokes operator.

Apart from \cite{avr} all mentioned articles use the same idea which is based on the mean value operator. 
In this article we present an alternative approach to the problem of existence of strong solutions to the Navier-Stokes equations. The motivation comes from \cite{zaj1} and \cite{zaj2}. The key tool is the refined energy estimate and the Sobolev embedding theorem. The essential part of the proof takes only a couple of lines but the price we pay is stronger assumption on the data and the forcing term. We should also note that from the final estimate for solution $v$ it follows that $v$ belongs to function space which is much smaller than required by the Serrin condition. It suggests that further improvements are possible.

At the beginning we mentioned that two types of the boundary conditions will be examined. Before we provide further clarification, let us introduce the following short-hand notation:
\begin{align*}
	\Omega_\epsilon &:= \Omega' \times (0,\epsilon), \\
	S &:= \partial \Omega, \\
	S &= S_B \cup S_T \cup S_L, \\
\intertext{where}
	S_B &:= \Omega' \times \{0\}, \\
	S_T &:= \Omega' \times \{\epsilon\}, \\
	S_L &:= \partial \Omega' \times (0,\epsilon)
\end{align*}
and $\Omega' \subset \mathbb{R}^2$ is a bounded and open subset with the boundary of $\mathcal{C}^2$. The subscripts $B$, $T$ and $L$ denote the bottom, the top and the lateral part of the boundary. We see that the domain $\Omega_\epsilon$ is of cylindrical type, which is placed alongside $x_3$-axis. The boundary conditions, we use to supplement problem \eqref{eq1} are of the form:
\begin{description}
	\item[\textbf{Case 1}] We put simply $v = 0$ on the whole boundary, which relaxes the restriction on the domain. The only requirement now is that it has to be thin in $x_3$ direction, which in particular implies that it is no longer assumed to be of cylindrical type.
	\item[\textbf{Case 2}] The domain to be considered is of cubical type, $\Omega_\epsilon = (0,l_1)\times(0,l_2)\times(0,\epsilon)$. On the top and the bottom we put $v = 0$, whereas on the side walls we assume periodicity with periods equal to $l_1$ and $l_2$.
\end{description}
The choice for boundary conditions is tightly linked to the Poincar\'e inequality, which is crucial in our approach. Although for other choices of boundary conditions the Poincar\'e inequality does not hold in general but it is still possible to prove the existence of strong solutions. We will demonstrate it in the forthcoming paper. 

\section{Auxiliary results}

\

We recall that in \cite[Prop. 2.1]{tem} the following Poincar\'e inequality was proved
\begin{prop}\label{prop1}
	Suppose that $v \in H^1(\Omega_\epsilon)$ satisfy one of the following conditions:
	\begin{equation*}
		\begin{aligned}
			&(i) & &v = 0 & &\text{on $S_B$}, \\
			&(ii) & &v = 0 & &\text{on $S_T$}.
		\end{aligned}
	\end{equation*}
	Then
	\begin{equation*}
		\norm{v}_{L_2(\Omega_\epsilon)} \leq \epsilon \norm{v_{,x_3}}_{L_2(\Omega_\epsilon)}.
	\end{equation*}
\end{prop}

We also need the Sobolev embedding $H^1 \hookrightarrow L_6$, which according to \cite[Rem. 2.1]{tem} is of the form
\begin{prop}
	Let $\Omega_\epsilon = \Omega'\times (0,\epsilon)$, where $\Omega' \in \mathcal{C}^2$. Then for all $v\in H^1(\Omega_\epsilon)$ we have
	\begin{equation*}
		\norm{v}_{L_6(\Omega_\epsilon)} \leq c_{\Omega'} \left(\frac{1}{\epsilon}\norm{v}_{L_2(\Omega_\epsilon)} + \norm{v_{,x_3}}_{L_2(\Omega_\epsilon)}\right)^{\frac{1}{3}}\left(\norm{v}_{L_2(\Omega_\epsilon)} + \norm{v_{,x_1}}_{L_2(\Omega_\epsilon)} + \norm{v_{,x_2}}_{L_2(\Omega_\epsilon)}\right)^{\frac{2}{3}}.
	\end{equation*}
\end{prop}

Combining these two above Propositions we get immediately
\begin{rem}\label{rem1}
	For all $v \in H^1(\Omega_\epsilon)$, $\Omega' \in \mathcal{C}^2$ the estimate
	\begin{equation}\label{eq5}
		\norm{v}_{L_6(\Omega_\epsilon)}^2 \leq c_{\Omega'} \norm{v}_{H^1(\Omega_\epsilon)}^2
	\end{equation}	
	holds.
\end{rem}
Next, we derive the fundamental energy estimate for weak solutions to \eqref{eq1}. This time we take into account the thickness of the domain along $x_3$-variable.
\begin{lem}\label{lem3}
	Suppose that
	\begin{equation*}
		\begin{aligned}
			v(t_0) &\in L_2(\Omega';L_{p_1}(0,\epsilon)) & p_1 &> 2, \\
			f &\in L_2(t_0,T;L_{\frac{6}{5}}(\Omega');L_{p_2}(0,\epsilon)) & p_2 &> \frac{6}{5}.
		\end{aligned}
	\end{equation*}
	Then $v \in L_{\infty}(t_0,T;L_2(\Omega_{\epsilon}))\cap L_2(t_0;H^1(\Omega_{\epsilon}))$ and
	\begin{equation*}
		\sup_{t \in (t_0,T)} \norm{v(t)}^2_{L_2(\Omega_{\epsilon})} + \nu\norm{\nabla v}^2_{L_2(\Omega^T_{\epsilon})} \leq \epsilon c_{\nu,\Omega'} \left(\norm{v(t_0)}_{L_2(\Omega';L_{p_1}(0,\epsilon))}^2 + \norm{f}_{L_2(t_0,T;L_{\frac{6}{5}}(\Omega');L_{p_2}(0,\epsilon))}^2\right)
	\end{equation*}
	holds.
\end{lem}

\begin{proof}
	We multiply \eqref{eq1}$_1$ by $v$ and integrate over $\Omega$. By \eqref{eq1}$_2$ and the boundary conditions we obtain
	\begin{equation*}
		\frac{1}{2}\Dt \int_{\Omega} \abs{v(t)}^2\, \ud x + \nu\norm{\nabla v}^2_{L_2(\Omega)} = \int_{\Omega} f\cdot v\, \ud x.
	\end{equation*}
	By the H\"older, the Young with $\epsilon$ inequalities and in view of Remark \ref{rem1} we see
	\begin{equation*}
		\int_{\Omega} f\cdot v\, \ud x \leq \norm{v}_{L_6(\Omega)}\norm{f}_{L_{\frac{6}{5}}(\Omega)} \leq \epsilon\norm{v}_{L_6(\Omega)}^2 + \frac{1}{4\epsilon}\norm{f}_{L_{\frac{6}{5}}(\Omega)}^2 \leq \epsilon c_{\Omega'}\norm{v}_{H^1(\Omega)}^2 + \frac{1}{4\epsilon}\norm{f}_{L_{\frac{6}{5}}(\Omega)}^2.
	\end{equation*}
	By Proposition \ref{prop1} we have
	\begin{equation*}
		c_{\Omega'}\norm{v}_{H^1(\Omega)}^2 \leq 2c_{\Omega'}\norm{\nabla v}_{L_2(\Omega)}^2.
	\end{equation*}
	Finally, for $\epsilon = \frac{\nu}{4c_{\Omega'}}$ we get
	\begin{equation*}
		\frac{1}{2}\Dt \int_{\Omega} \abs{v(t)}^2\, \ud x + \frac{\nu}{2}\norm{\nabla v}^2_{L_2(\Omega)} \leq \frac{c_{\Omega'}}{\nu} \norm{f}_{L_{\frac{6}{5}}(\Omega)}^2.
	\end{equation*}
	Multiplying by $2$ and integrating with respect to $t$ gives
	\begin{equation*}
		\sup_{t \in (t_0,T)} \norm{v(t)}^2_{L_2(\Omega)} + \nu\norm{\nabla v}^2_{L_2(\Omega^t)} \leq \frac{2c_{\Omega'}}{\nu} \norm{f}_{L_2(t_0,T;L_{\frac{6}{5}}(\Omega))}^2 + \norm{v(t_0)}^2_{L_2(\Omega)}.
	\end{equation*}
	From the assumption and the H\"older inequality it follows that
	\begin{multline*}
		\norm{v(t_0)}^2_{L_2(\Omega)} = \int_{\Omega} v_{t_0}^2\, \ud x = \int_{\Omega'}\int_0^{\epsilon} v_{t_0}^2(x',x_3)\, \, \ud x_3\ud x' \\
		\leq \int_{\Omega'}\left(\int_0^{\epsilon} v_{t_0}^{2p}(x',x_3)\, \, \ud x_3\right)^{\frac{1}{p}} \left(\int_0^{\epsilon} 1\, \ud x_3\right)^{\frac{1}{q}}\ud x' = \epsilon^{\frac{1}{q}} \norm{v(t_0)}_{L_2(\Omega';L_{2p}(0,\epsilon))}^2,
	\end{multline*}
	where $\frac{1}{p} + \frac{1}{q} = 1$, $p > 1$.
	In the same manner
	\begin{equation*}
		\norm{f}_{L_2(t_0,T;L_{\frac{6}{5}}(\Omega))}^2 \leq \epsilon^{\frac{1}{q}}\norm{f}_{L_2(t_0,T;L_{\frac{6}{5}}(\Omega');L_{\frac{6}{5}p}(0,\epsilon))}^2.
	\end{equation*}
	This ends the proof.
\end{proof}

The last tool is purely technical:
\begin{lem}\label{lem1}
	Let $\Omega \subset \mathbb{R}^3$ be a bounded and open subset. Suppose that $\nabla v_{,t} \in L_2(\Omega^t)$. Then $\nabla v \in L_\infty(0,t;L_2(\Omega))$ and the inequality
	\begin{equation*}
		\norm{\nabla v}^2_{L_\infty(0,t;L_2(\Omega))} \leq \norm{\nabla v}^2_{L_2(\Omega^t)} + \norm{\nabla v_{,t}}^2_{L_2(\Omega^t)} + \norm{\nabla v(0)}^2_{L_2(\Omega)}
	\end{equation*}
	holds.
\end{lem}

\begin{proof}
	Observe that
	\begin{equation*}
		\Dt \int_{\Omega} \nabla v\cdot \nabla v\, \ud x \leq \abs{\Dt \int_{\Omega} \nabla v\cdot \nabla v\, \ud x} = 2 \abs{\int_{\Omega} \nabla v\cdot \nabla v_{,t}\, \ud x} \leq \norm{\nabla v}^2_{L_2(\Omega)} + \norm{\nabla v_{,t}}^2_{L_2(\Omega^t)}.
	\end{equation*}
	Integrating with respect to time ends the proof.
\end{proof}

\section{Proof of Theorem \ref{t1}}
Using Lemma \ref{lem1} and the assertion of Remark \ref{rem1} we can give the proof of Theorem \ref{t1}.

\begin{proof}[Proof of Theorem \ref{t1}]
	First we differentiate \eqref{eq1} with respect to time. It gives
	\begin{equation}\label{eq2}
		\begin{aligned}
			&v_{,tt} + v_{,t}\cdot \nabla v + v\cdot \nabla v_{,t} - \nu \triangle v_{,t} + \nabla p_{,t} = f_{,t} & &\text{in $\Omega_\epsilon^t$}, \\
			&\Div v_{,t} = 0 & &\text{in $\Omega_\epsilon^t$}.
		\end{aligned}
	\end{equation}
	Multiplying \eqref{eq2}$_1$ by $v_{,t}$ and integrating by parts and using \eqref{eq1}$_2$ gives
	\begin{equation*}
		\frac{1}{2}\Dt \int_{\Omega_\epsilon} \abs{v_{,t}}^2\, \ud x + \nu \int_{\Omega_\epsilon}\abs{\nabla v_{,t}}^2\, \ud x = \int_{\Omega_\epsilon} f_{,t}\cdot v_{,t}\, \ud x - \int_{\Omega_\epsilon} v_{,t}\cdot \nabla v \cdot v_{,t}\, \ud x.
	\end{equation*}
	To estimate the first term on the right hand side we use the H\"older and the Young inequalities
	\begin{equation}\label{eq3}
		 \int_{\Omega_\epsilon} f_{,t}\cdot v_{,t}\, \ud x \leq \norm{f_{,t}}_{L_{\frac{6}{5}}(\Omega_\epsilon)}\norm{v_{,t}}_{L_6(\Omega_\epsilon)} \leq \epsilon_1\norm{v_{,t}}_{L_6(\Omega_\epsilon)}^2 + \frac{1}{4\epsilon_1}\norm{f_{,t}}_{L_{\frac{6}{5}}(\Omega_\epsilon)}^2.
	\end{equation}
	For the second term we use the H\"older inequality
	\begin{equation*}
		\int_{\Omega_\epsilon} v_{,t}\cdot \nabla v \cdot v_{,t}\, \ud x \leq \norm{v_{,t}}_{L_4(\Omega_\epsilon)}\norm{\nabla v}_{L_2(\Omega_\epsilon)}\norm{v_{,t}}_{L_4(\Omega_\epsilon)}.
	\end{equation*}
	For $L_p$-spaces we have the interpolation inequality
	\begin{equation*}
		\norm{v_{,t}}_{L_r(\Omega_\epsilon)} \leq \norm{v_{,t}}^\theta_{L_s(\Omega_\epsilon)}\norm{v_{,t}}^{1 - \theta}_{L_p(\Omega_\epsilon)},
	\end{equation*}
	where $s \leq r \leq p$ and
	\begin{equation*}
		\frac{1}{r} = \frac{\theta}{s} + \frac{1 - \theta}{p}.
	\end{equation*}
	Setting $r = 4$, $s = 2$ and $p = 6$ we obtain $\theta = \frac{1}{4}$, which justifies the inequality
	\begin{equation*}
		\norm{v_{,t}}_{L_4(\Omega_\epsilon)}\norm{\nabla v}_{L_2(\Omega_\epsilon)}\norm{v_{,t}}_{L_4(\Omega_\epsilon)} \leq \norm{v_{,t}}_{L_2(\Omega_\epsilon)}^{\frac{1}{2}}\norm{v_{,t}}_{L_6(\Omega_\epsilon)}^{\frac{3}{2}}\norm{\nabla v}_{L_2(\Omega_\epsilon)}.
	\end{equation*}
	From the Young inequality it follows that
	\begin{equation}\label{eq4}
		\norm{v_{,t}}_{L_2(\Omega_\epsilon)}^{\frac{1}{2}}\norm{v_{,t}}_{L_6(\Omega_\epsilon)}^{\frac{3}{2}}\norm{\nabla v}_{L_2(\Omega_\epsilon)} \leq \epsilon_2 \norm{v_{,t}}_{L_6(\Omega_\epsilon)}^2 + \frac{1}{4\epsilon_2} \norm{v_{,t}}_{L_2(\Omega_\epsilon)}^2\norm{\nabla v}_{L_2(\Omega_\epsilon)}^4.
	\end{equation}
	From \eqref{eq3}, \eqref{eq4} and in view of Remark \ref{rem1} we get that
	\begin{equation*}
		\Dt \int_{\Omega_\epsilon} \abs{v_{,t}}^2\, \ud x + \nu \int_{\Omega_\epsilon}\abs{\nabla v_{,t}}^2\, \ud x \leq c_{\nu,\Omega'} \norm{f_{,t}}_{L_{\frac{6}{5}}(\Omega_\epsilon)}^2 + c_{\nu,\Omega'} \norm{v_{,t}}_{L_2(\Omega_\epsilon)}^2\norm{\nabla v}_{L_2(\Omega_\epsilon)}^4.
	\end{equation*}
	Integrating with respect to $t \in (t_0,T)$ yields
	\begin{multline}\label{eq13}
		\sup_{t\in (t_0,T)} \norm{v_{,t}(t)}^2_{L_2(\Omega_\epsilon)} + \nu \norm{\nabla v_{,t}}^2_{L_2(\Omega_\epsilon^t)} \\
		\leq c_{\nu,\Omega'} \norm{f_{,t}}_{L_2(t_0,T;L_{\frac{6}{5}}(\Omega_\epsilon))}^2 + c_{\nu,\Omega'} \int_{t_0}^T\norm{v_{,t}(t)}_{L_2(\Omega_\epsilon)}^2\norm{\nabla v(t)}_{L_2(\Omega_\epsilon)}^4\, \ud t + \norm{v_{,t}(t_0)}^2_{L_2(\Omega_\epsilon)}.
	\end{multline}
	Next we see
	\begin{multline*}
		c_{\nu,\Omega'} \int_{t_0}^T\norm{v_{,t}(t)}_{L_2(\Omega_\epsilon)}^2\norm{\nabla v(t)}_{L_2(\Omega_\epsilon)}^4\, \ud t = c_{\nu,\Omega'} \int_{t_0}^T\norm{v_{,t}(t)}_{L_2(\Omega_\epsilon)}^2\norm{\nabla v(t)}_{L_2(\Omega_\epsilon)}^2\norm{\nabla v(t)}_{L_2(\Omega_\epsilon)}^2\, \ud t \\
		\leq c_{\nu,\Omega'} \sup_{t \in (t_0,T)}\norm{v_{,t}(t)}_{L_2(\Omega_\epsilon)}^2\norm{\nabla v}^2_{L_{\infty}(t_0,T;L_2(\Omega_\epsilon))} \int_{t_0}^T\norm{\nabla v(t)}_{L_2(\Omega_\epsilon)}^2\, \ud t.
	\end{multline*}
	By Lemmas \ref{lem3} and \ref{lem1} we rewrite \eqref{eq13} in the form
	\begin{multline*}
		\sup_{t\in (t_0,T)} \norm{v_{,t}(t)}^2_{L_2(\Omega_\epsilon)} + \nu \norm{\nabla v}^2_{L_{\infty}(t_0,T;L_2(\Omega^t))} \\
		\leq \epsilon c_{\nu,\Omega'} \sup_{t \in (t_0,T)}\norm{v_{,t}(t)}_{L_2(\Omega_\epsilon)}^2\norm{\nabla v}^2_{L_{\infty}(t_0,T;L_2(\Omega_\epsilon))} \left(\norm{v(t_0)}_{L_2(\Omega';L_{p_1}(0,\epsilon))}^2 + \norm{f}_{L_2(t_0,T;L_{\frac{6}{5}}(\Omega');L_{p_2}(0,\epsilon))}^2\right) \\
		  + c_{\nu,\Omega'} \norm{f_{,t}}_{L_2(t_0,T;L_{\frac{6}{5}}(\Omega_\epsilon))}^2 + \norm{v_{,t}(t_0)}^2_{L_2(\Omega_\epsilon)} + \nu\norm{\nabla v}_{L_2(\Omega_\epsilon^t)}^2 + \nu\norm{\nabla v(t_0)}^2_{L_2(\Omega_\epsilon)}.
	\end{multline*}
	Taking $\epsilon$ sufficiently small (see Remark below) ends the proof.
\end{proof}

\begin{rem}
	In the last Lemma we dealt with an inequality of the form
	\begin{equation*}
		x + y \leq \epsilon c x y + b,
	\end{equation*}
	where $x, y, b > 0$. We claim that for $\epsilon > 0$ sufficiently small the above inequality reduces to
	\begin{equation*}
		x + y \leq  b.
	\end{equation*}
	To prove it let us simplify the first inequality. Let $u := x + y$. Next, we utilize the Cauchy inequality on the right-hand side which gives
	\begin{equation*}
		u \leq \epsilon c x y + b \leq \frac{\epsilon c}{2}\left(x^2 + y^2\right) + b \leq \epsilon c u^2 + b.
	\end{equation*}
	We will show that for certain choice of $\epsilon < \epsilon*$ the function $u$ satisfies
	\begin{equation*}
		u = \epsilon c u^2 + b
	\end{equation*}
	or equivalently that $u$ is a fixed point to the mapping
	\begin{equation*}
		z = \epsilon c f(u) + b,
	\end{equation*}
	where $f(u) = u^2$. Indeed, let $u_n$ be a bounded sequence (we assume that a solution to \eqref{eq1} exists) such that
	\begin{equation*}
		u_{n + 1} = \epsilon c f(u_n) + b.
	\end{equation*}
	To check that $u_n$ satisfies the Cauchy condition we consider the difference $u_{n +1} - u_n$. We have
	\begin{multline*}
		\abs{u_{n + 1} - u_n} = \epsilon c\abs{f(u_n) - f(u_{n - 1})} = \epsilon c \abs{f'(u^*_n)} \abs{u_n - u_{n - 1}} \leq \epsilon c M\abs{u_n - u_{n - 1}} \\
		= \epsilon^2 c^2 M\abs{f(u_{n - 1}) - f(u_{n - 2})} = \epsilon^2 c^2 M \abs{f'(u^*_{n - 1})}\abs{u_{n - 1} - u_{n - 2}} \leq \ldots \leq \epsilon^n c^nM^n \abs{u_1 - u_0}.
	\end{multline*}
	Hence, for $\epsilon < \frac{1}{cM}$ the sequence $u_n$ is convergent to $u$ for $n \to \infty$. This ends the proof and Remark.
\end{rem}

\end{document}